\documentclass [12pt,a4paper,reqno]{amsart}
\usepackage{amsmath,amsfonts,amssymb,amscd,calrsfs,verbatim,delarray}
\usepackage[arrow,matrix]{xy}

\DeclareMathOperator{\tr}{tr}
\DeclareMathOperator{\Sz}{Sz}
\DeclareMathOperator{\PSL}{PSL}

\DeclareMathOperator{\SL}{SL}
\DeclareMathOperator{\SU}{SU}
\DeclareMathOperator{\Sp}{Sp}

\DeclareMathOperator{\Irr}{Irr}

\chardef\bslash=`\\ 

\numberwithin{equation}{section}
\newtheorem{theorem}{Theorem}[section]
\newtheorem{prop}[theorem]{Proposition}
\newtheorem{lemma}[theorem]{Lemma}
\newtheorem{corr}[theorem]{Corollary}
\newtheorem{conj}[theorem]{Conjecture}
\newtheorem{property}[theorem]{Properties}

\newtheorem{theoremS}{Theorem}

\newtheorem{corrS}[theoremS]{Corollary}

\theoremstyle{definition}

\newtheorem{defn}[theorem]{Definition}
\newtheorem{quest}[theorem]{Question}
\newtheorem{example}[theorem]{Example}
\newtheorem{remark}[theorem]{Remark}

\newcommand{\thmref}[1]{Theorem~\ref{#1}}
\newcommand{\secref}[1]{Section~\ref{#1}}

\newcommand{\lemref}[1]{Lemma~\ref{#1}}
\newcommand{\corref}[1]{Corollary~\ref{#1}}

\newcommand{\remarkref}[1]{Remark~\ref{#1}}

\newcommand{\defnref}[1]{Definition~\ref{#1}}
\newcommand{\propref}[1]{Proposition~\ref{#1}}
\newcommand{\prref}[1]{Properties~\ref{#1}}
\newcommand{\conref}[1]{Conjecture~\ref{#1}}

\newcommand{\beq}{\begin{equation}}
\newcommand{\ef}{\end{equation}}

\newcommand{\BP}{\mathbb{P}}

\newcommand{\BN}{\mathbb{N}}
\newcommand{\BZ}{\mathbb{Z}}\newcommand{\BF}{\mathbb{F}}
\newcommand{\BA}{\mathbb {A}}

\newcommand{\st}{\sigma}
\renewcommand{\k}{\varkappa}
\renewcommand{\th}{\theta}

\newcommand{\g}{\gamma}
\newcommand{\be}{\beta}
\newcommand{\dl}{\delta}
\newcommand{\om}{\omega}
\newcommand{\vp}{\varphi}

\renewcommand{\a}{\alpha}

\newcommand{\ep}{\epsilon}

\newcommand{\G}{\Gamma}

\newcommand{\ov}{\overline}

\newcommand{\tG}{\widetilde G}

\newcommand{\de}{\partial}
\newcommand{\fc}{\frac}

\newcommand{\fr}{\frac}
\begin{document}

\title{On the Surjectivity of Engel Words on $\PSL(2,q)$}

\author{Tatiana Bandman, Shelly Garion and Fritz Grunewald}

\address{Bandman: Department of Mathematics, Bar-Ilan University, 52900 Ramat Gan, ISRAEL}
\email{bandman@macs.biu.ac.il}

\address{Garion: Institut des Hautes \'{E}tudes Scientifiques, 91440 Bures-sur-Yvette, FRANCE}
\email{shellyg@ihes.fr}





\begin{abstract}
We investigate the surjectivity of the word map defined by the
$n$-th Engel word on the groups $\PSL(2,q)$ and $\SL(2,q)$. For
$\SL(2,q)$, we show that this map is surjective onto the subset
$\SL(2,q)\setminus\{-id\}\subset \SL(2,q)$ provided that $q \geq
q_0(n)$ is sufficiently large. Moreover, we give an estimate for
$q_0(n)$. We also present examples demonstrating that this does not
hold for all $q$.

We conclude that the $n$-th Engel word map is surjective for the
groups $\PSL(2,q)$ when $q \geq q_0(n)$. By using the computer, we
sharpen this result and show that for any $n \leq 4$, the
corresponding map is surjective for \emph{all} the groups
$\PSL(2,q)$. This provides evidence for a conjecture of Shalev
regarding Engel words in finite simple groups.

In addition, we show that the $n$-th Engel word map is almost
measure preserving for the family of groups $\PSL(2,q)$, with $q$
odd, answering another question of Shalev.

Our techniques are based on the method developed by Bandman,
Grunewald and Kunyavskii for verbal dynamical systems in the group
$\SL(2,q)$.
\end{abstract}

\maketitle

\section{Introduction}\label{intro}

\subsection{Word maps in finite simple groups}

During the last years there was a great interest in \emph{word maps} in groups
(for an extensive survey see \cite{Se}). These maps are defined as follows.
Let $w=w(x_1,\dots,x_d)$ be a non-trivial \emph{group word}, namely a non-identity element of the free group
$F_d$ on $x_1,\dots,x_d$.
Then we may write $w=x_{i_1}^{n_1}x_{i_2}^{n_2}\dots x_{i_k}^{n_k}$ where $1 \leq i_j \leq d$, $n_j \in \BZ$,
and we may assume further that $w$ is reduced.
Let $G$ be a group. For $g_1,\dots,g_d$ we write
$$
   w(g_1,\dots,g_d)=g_{i_1}^{n_1}g_{i_2}^{n_2}\dots g_{i_k}^{n_k} \in G,
$$
and define
$$
   w(G) = \{w(g_1,\dots,g_d): g_1,\dots,g_d \in G \},
$$
as the set of values of $w$ in $G$.
The corresponding map $w:G^d \rightarrow G$ is called a \emph{word map}.

It is interesting to estimate the size of $w(G)$. Borel \cite{Bo}
showed that the word map induced by $w \neq 1$ on simple algebraic
groups is a dominant map. Larsen \cite{La} used this result to show
that for every non-trivial word $w$ and $\ep>0$ there exists a
number $C(w,\ep)$ such that if $G$ is a finite simple group with
$|G|>C(w,\ep)$ then $|w(G)| \geq |G|^{1-\ep}$. By a celebrated
result of Shalev \cite{Sh09} one has that for every non-trivial word
$w$ there exists a constant $C(w)$ such that if $G$ is a finite
simple group satisfying $|G|> C(w)$ then $w(G)^3=G$. These results
were substantially improved by Larsen and Shalev \cite{LS} for
various families of finite simple groups, and have recently been
generalized by Larsen, Shalev and Tiep \cite{LST}.

One can therefore ask whether $w(G)=G$ for any non-trivial word $w$
and all finite simple non-abelian groups $G$. The answer to this
question is clearly negative. It is easy to see that if $G$ is a
finite group and $m$ is an integer which is not relatively prime to
the order of $G$ then for the word $w=x_1^m$ one has that $w(G) \neq
G$. Hence, if $v \in F_d$ is any word, then the word map
corresponding to $w=v^m$ cannot be surjective. A natural question,
suggested by Shalev, is whether these words are generally the only
exceptions for non-surjective word maps in finite simple non-abelian
groups. In particular, the following conjecture was raised:

\begin{conj}[Shalev]\cite[Conjectures 2.8 and 2.9]{Sh07}. \label{conj.surjall}
Let $w \ne 1$ be a word which is not a proper power of another word.
Then there exists a number $C(w)$ such that, if $G$ is either $A_r$
or a finite simple group of Lie type of rank $r$, where $r > C(w)$,
then $w(G)=G$.
\end{conj}

It is now known that for the commutator word $w=[x,y] \in F_2$, one
has that $w(G)=G$ for any finite simple non-abelian group $G$. This
statement is the well-known \emph{Ore Conjecture}, originally posed
in 1951 and proved by Ore himself for the alternating groups
\cite{Or}. During the years, this conjecture was proved for various
families of finite simple groups (see \cite{LOST} and the references
therein). Thompson \cite{Th} established it for the groups
$\PSL(n,q)$, later Ellers and Gordeev \cite{EG} proved the
conjecture for all finite simple groups of Lie type defined over a
field with more than $8$ elements, and recently the proof was
completed for all finite simple groups in a celebrated work of
Liebeck, O'Brien, Shalev and Tiep \cite{LOST}.

There was also an interest in quasisimple groups. By \cite{Th} and
\cite{LOST}, in every quasisimple classical group $\SL(n,q)$,
$\SU(n,q)$, $\Sp(n,q)$, $\Omega^{\pm}(n,q)$, every element is a
commutator (a \emph{quasisimple} group $G$ is a perfect group such
that $G/Z(G)$ is simple). However it is not true that every element
of every quasisimple group is a commutator, see the examples in
\cite{Bl}.


\subsection{Engel words}

After considering the commutator word, it is natural to consider the Engel words.
These words are defined recursively as follows.

\begin{defn}
The \emph{$n$-th Engel word} $e_n(x,y) \in F_2$ is defined recursively by
\begin{align*}
  e_1(x,y)&= [x,y] = xyx^{-1}y^{-1} , \\
  e_n(x,y)&= [e_{n-1},y] , \text{ for } n>1 .
\end{align*}
For a group $G$, the corresponding map $e_n: G \times G \rightarrow G$ is called
the \emph{$n$-th Engel word map}.
\end{defn}

Now, the following conjecture is naturally raised.

\begin{conj}[Shalev]\label{shalev}
Let $n \in \BN$, then the $n$-th Engel word map is surjective for
any finite simple non-abelian group $G$.
\end{conj}

For some (small) finite simple non-abelian groups this conjecture
was verified by O'Brien using the \textsc{Magma} computer program.

Note that in order to complete the proof of Ore's Conjecture,
Liebeck, O'Brien, Shalev and Tiep used the classical criterion
dating back to Frobenius, characterizing the possibility of writing
an element $g$ in a finite group $G$ as a commutator by the
non-vanishing of the character sum
$$\sum_{\chi \in \Irr(G)} \frac{\chi(g)}{\chi(1)},$$
(see \cite{LOST} and the references therein). Unfortunately, it is
unknown whether there is an analogous criterion for the possibility
of writing an element as an Engel word $e_n$, $n>1$. Hence, Shalev's
Conjecture seems to be substantially more difficult than Ore's
Conjecture, even for certain families of finite simple groups, such
as $\PSL(2,q)$.


\subsection{Engel words in $\PSL(2,q)$ and $\SL(2,q)$}

We consider Engel words in the particular case of the groups $\PSL(2,q)$ and $\SL(2,q)$,
in an attempt to prove \conref{shalev} for the group $\PSL(2,q)$.

By Thompson \cite{Th}, every element of $\SL(n,q)$, except when
$(n,q)=(2,2),(2,3)$, is a commutator (including the central
elements). Moreover, Blau \cite{Bl} proved that with a few specified
exceptions, every central element of a finite quasisimple group is a
commutator. In particular, if $G$ is a quasisimple group of simply
connected Lie type, then every element of $Z(G)$ is a commutator.
Interestingly, such a result fails to hold for Engel words.

Indeed, in the group $\SL(2,q)$, where $q$ is odd, if $n\geq n_0(q)$ is large enough then
the central element $-id$ cannot be written as an $n$-th Engel word, namely
$e_n(x,y)\ne -id$ for any $x,y \in \SL(2,q)$ (see \propref{prop.minid}), implying that
the $n$-th Engel word map is not surjective.
This leads us to introduce the following notion of ``almost surjectivity''.

\begin{defn}\label{almostsurj}
A word map $w: \SL(2,q)^d\to \SL(2,q)$ is \emph{almost surjective}
if $w(\SL(2,q))=\SL(2,q)\setminus \{-id\}.$
\end{defn}

A method for investigating verbal dynamical systems in the group $\SL(2,q)$, using the
so-called \emph{trace map}, was introduced in \cite{BGK}.
We use this method to study the dynamics of the trace map instead of
solving equations in groups. There is a special property of the
Engel word $e_n(x,y)$ which makes the dynamics of the trace map
particularly amenable to analysis: for a group $G$ the morphism
$G^2\to G^2$ defined as $(x,y)\mapsto (e_n(x,y),y)$ is not dominant.
Using this method we obtain the following result.

\begin{theoremS}\label{thm.engel}
Let $n \in \BN$, then the $n$-th Engel word map is almost surjective
for the group $\SL(2,q)$ provided that $q\geq q_0(n)$ is
sufficiently large.
\end{theoremS}
We moreover give an estimate for $q_0(n)$ which, unfortunately, is
exponential in $n$ (see \corref{less}).

\vskip 0.2 cm

\thmref{thm.engel} certainly fails to hold for all groups
$\SL(2,q)$. Indeed, we give examples for integers $n \geq 3$ and
finite fields $\BF_q$ for which the $n$-th Engel word map is
\emph{not} almost surjective for $\SL(2,q)$ (see Example
\ref{ex.no.sol}). We moreover show that there is an infinite family
of finite fields $\BF_q$, such that if $n\geq n_0(q)$ is large
enough, then the $n$-th Engel word map in not almost surjective on
$\SL(2,q)$ (see \propref{sne1}).

\vskip 0.2 cm

Considering the group $\PSL(2,q)$, we see that \thmref{thm.engel}
immediately implies that the $n$-th Engel word map is surjective for
the group $\PSL(2,q)$ provided that $q\geq q_0(n)$. Thus, when $n$
is small, one can verify by computer that the $n$-th Engel word map
is surjective for the remaining groups $\PSL(2,q)$ with $q<q_0(n)$,
hence for all the groups $\PSL(2,q)$.

\begin{corrS}\label{cor.short}
Let $n \leq 4$ then the $n$-th Engel word map is surjective for all
groups $\PSL(2,q)$.
\end{corrS}

We have moreover shown that there are certain infinite families of
finite fields $\BF_q$ for which the $n$-th Engel word map in
$\PSL(2,q)$ is always surjective for every $n \in \BN$. The first
family consists of all finite fields of characteristic $2$ (see
\propref{prop.even}), and the second family contains infinitely many
finite fields of odd characteristic (see \propref{prop.sqrt}).
Following \conref{shalev} we believe that the surjectivity should in
fact hold for all groups $\PSL(2,q)$.


\subsection{Equidistribution and measure preservation}

Another interesting question is the \emph{distribution} of a word map.
For a word $w=w(x_1,\dots,x_d) \in F_d$, a finite group $G$ and some $g \in G$, we define
\[
    N_w(g) = \{(g_1,\dots,g_d)\in G^d: w(g_1,\dots,g_d)=g\}.
\]
It is therefore interesting to estimate the size of $N_w(g)$, and especially to see whether
$w$ is \emph{almost equidistributed}, namely if $|N_w(g)| \approx |G|^{d-1}$ for almost all $g \in G$.
More precisely, we define:

\begin{defn}\label{def.equi}
A word map $w:G^d \rightarrow G$ is \emph{almost equidistributed}
for a family of finite groups $\mathcal{G}$ if any group $G \in
\mathcal{G}$ contains a subset $S= S_G \subseteq G$ with the
following properties:
\begin{enumerate}\renewcommand{\theenumi}{\it \roman{enumi}}
\item $|S| = |G|(1-o(1))$;
\item $|N_w(g)| = |G|^{d-1}(1+o(1))$ uniformly for all $g \in S$;
\end{enumerate}
where $o(1)$ denotes a real number depending only on $G$ which tends to zero as $|G| \rightarrow \infty$.
\end{defn}

An important consequence (see \cite[\S 3]{GS}) is that any ``almost
equidistributed'' word map is also ``almost measure preserving'',
that is:

\begin{defn}\label{def.measure}
A word map $w:G^d \rightarrow G$ is \emph{almost measure preserving}
for a family of finite groups $\mathcal{G}$ if every group $G
\in \mathcal{G}$ satisfies the following:
\begin{enumerate}\renewcommand{\theenumi}{\it \roman{enumi}}
\item For every subset $Y \subseteq G$ we have
$$ |w^{-1}(Y)|/|G|^d = |Y|/|G|+ o(1); $$
\item For every subset $X \subseteq G^d$ we have
$$ |w(X)|/|G| \geq |X|/|G|^d - o(1); $$
\item In particular, if $X \subseteq G^d$ and $|X|/|G|^d = 1-o(1)$,
then almost every element $g \in G$ can be written as $g=w(g_1,\dots,g_d)$ where $g_1,\dots,g_d \in X$;
\end{enumerate}
where $o(1)$ denotes a real number depending only on $G$ which tends to zero as $|G| \rightarrow \infty$.
\end{defn}

The following question was raised by Shalev.

\begin{quest}[Shalev]\cite[Problem 2.10]{Sh07}.
Which words $w$ induce almost measure preserving word maps $w: G^d \rightarrow G$
on finite simple groups $G$?
\end{quest}

It was proved in \cite{GS} that the commutator word $w=[x,y] \in
F_2$ as well as the words $w =[x_1,\dots,x_d] \in F_d$, $d$-fold
commutators in any arrangement of brackets, are almost
equidistributed, and hence also almost measure preserving, for the
family of finite simple non-abelian groups.

A natural question, suggested by Shalev, is whether this remains
true also for the Engel words. We prove that this is indeed true for
the family of groups $\PSL(2,q)$, where $q$ is odd.

\begin{theoremS}\label{thm.equi}
Let $n \in \BN$, then the $n$-th Engel word map is almost
equidistributed, and hence also almost measure preserving, for the
family of groups $\{\PSL(2,q):\ q \ is \ odd\}$.
\end{theoremS}

Since it is well-known that almost all pairs of elements in $\PSL(2,q)$ are
generating pairs (see \cite{KL}),
we deduce that for any $n \in \BN$, the probability that a randomly chosen element $g \in \PSL(2,q)$,
where $q$ is odd, can be written as an Engel word $e_n(x,y)$ where $x,y$ generate $\PSL(2,q)$,
tends to $1$ as $q \rightarrow \infty$.

It was proved in \cite{MW} that when $q \geq 13$ is odd, every
nontrivial element of $\PSL(2,q)$ is a commutator of a generating
pair. One can therefore ask if a similar result also holds for the
Engel words.

\subsection {Notation and layout}

Throughout the paper we use the following notation:

\begin{itemize}
\item $G=\PSL(2,q);$
\item $\tG=\SL(2,q);$
\item $\ov{\BF}_q$ -- the algebraic closure of the finite field
$\BF_q$;
\item $|M|$ -- number of points in a set $M;$
\item $\BA_{x_1,...,x_k}^k$ -- $k$-dimensional affine space with coordinates $x_1,...,x_k;$
\item  $p(s,u,t)=s^2+t^2+u^2-sut-2;$
\item $d(X)$ -- degree of a projective set $X;$
\item $g(X)$ --  geometric genus of a projective curve $X;$
\item $f^{(n)}$ stands for $n^{th}$ iteration of a morphism $f.$
\end{itemize}

Some words on the layout of this paper. In \secref{sect.trace} we recall the general method developed in
\cite{BGK} for investigating verbal systems in the group $\SL(2,q)$.
We apply this method to Engel words in \secref{sect.tr.engel}.
In \secref{special.fields} we discuss the surjectivity (and non-surjectivity) of Engel words
in the groups $\SL(2,q)$ and $\PSL(2,q)$ for certain families of finite fields.
The proof of our main theorem, \thmref{thm.engel}, appears in \secref{main}.
In \secref{small.n} we check the surjectivity of short Engel words for all groups $\PSL(2,q)$ and prove
\corref{cor.short}.
The proof of the equidistribution theorem, \thmref{thm.equi}, appears in \secref{sect.equi}.
In \secref{last} we discuss further questions and conjectures.

\section{The trace map}\label{sect.trace}

The main idea is to use the method which was introduced in~\cite{BGK} to investigate
verbal dynamical systems.
This method is based on the following classical Theorem (see, for example, \cite{Vo,Fr,FK} or
\cite{Ma, Go} for a more modern exposition).

\begin{theorem}[Trace map]\label{trace}
Let $F=\left<x,y\right>$ denote the free group on two generators.
Let us embed $F$ into $\SL(2,\BZ)$ and denote by $\tr$ the trace
character. If $w$ is an arbitrary element of $F$, then the character
of $w$ can be expressed as a polynomial
$$
\tr(w)=P(s,u,t)
$$
with integer coefficients in the three characters $s=\tr(x),
u=\tr(xy)$ and $t=\tr(y)$.
\end{theorem}

Note that the same remains true for the group $\tG =\SL(2,q)$. The general case, $\SL(2,R)$,
where $R$ is a commutative ring, can be found in \cite{CMS}.

\vskip 0.3 cm

The construction used below is described in detail in \cite{BGK}.
In this construction, $\SL(2,\ov{\BF}_q)$ is considered as an affine variety,
which we shall denote by $\tG$ as well, since no confusion may arise.
We will also consider $\SL(2,\BF_q)$ as a special fiber at $q$ of a $\BZ$-scheme $\SL(2,\BZ).$

For any $x, y\in \tG$ denote $s=\tr(x)$, $t=\tr(y)$ and
$u=\tr(xy)$, and define a morphism $\pi\colon \tG\times \tG\to
\BA_{s,u,t}^3$ by
$$\pi(x,y):=(s,u,t).$$

\begin{theorem}\cite[Theorem 3.4]{BGK}.\label{thm.main}
For every $\mathbb{F}_q$-rational point $Q=(s_0,u_0,t_0) \in
\mathbb{A}^3_{s,u,t}$, the fiber $H=\pi^{-1}(Q)$ has an
$\mathbb{F}_q$-rational point.
\end{theorem}

Let $\om(x,y)$ be a word in two variables and let
$\tilde\varphi: \tG \times \tG \rightarrow \tG$
be a morphism, defined as $\tilde\varphi(x,y)=\om(x,y).$

The \emph{Trace map Theorem} implies that there exists a morphism $\psi \colon
\BA_{s,u,t}^3\to \BA_{s,u,t}^3$ such that
\begin{equation}\label{cc1}
\psi(\pi(x,y))=\pi(\tilde\vp(x,y),y).
\end{equation}

This map is called the ``trace map'', and it satisfies
\begin{equation}\label{cc2}
\psi(s,u,t):= \bigl(f_1(s,u,t),f_2(s,u,t),t \bigr),
\end{equation}
where $f_1(s,u,t) = \tr(\tilde\vp(x,y))$ and $f_2(s,u,t) = \tr(\tilde\vp(x,y)y)$.

Define $\vp =(\tilde\vp, id)\colon \tG\times \tG\to \tG\times \tG$
by $\vp(x,y)=(\tilde\vp(x,y),y)$. Then, according to \eqref{cc1} and \eqref{cc2},
the following diagram commutes:
\begin{equation}
\begin{CD}
\tG \times \tG  @>{\vp}>> \tG \times \tG   \\
@V\pi VV   @V\pi VV\\
\BA^3_{s,u,t} @>{\psi}>> \BA^3_{s,u,t} \label{d4}
\end{CD}
\end{equation}

Therefore, the main idea is to study the properties of the morphism $\psi$ instead of
the corresponding word map $\om$.

As will be shown later, the morphism $\psi$ corresponding to Engel
words is much simpler. Moreover, it follows from \thmref{thm.main}
that the  surjectivity of $\psi$ implies the surjectivity of $\vp$
(see \propref{if}).


\section{Trace maps of Engel words}\label{sect.tr.engel}

Let $e_n = e_n(x,y): \tG \times \tG \rightarrow \tG$
be the $n$-th Engel word map, and let $s_n=\tr(e_n(x,y)).$

Then
\beq\label{t2}
s_1=\tr (e_1(x,y))=\tr ([x,y])=s^2+t^2+u^2-ust-2=p(s,u,t).
\ef

Moreover, for $n \geq 1$, \beq\label{t1}
  \tr(e_{n}(x,y)y) = \tr(e_{n-1}ye_{n-1}^{-1}y^{-1}y) = \tr(e_{n-1}ye_{n-1}^{-1}) = \tr(y) = t.
\ef

Therefore, for $n \geq 1$,
\beq\label{t3}
s_{n+1}=\tr(e_{n+1})=p(s_n,t,t)=s_n^2-s_nt^2+2t^2  - 2.
\ef

In the notation of diagram \eqref{d4} we have
\beq\label{psi}
 \psi(s,u,t) =(p(s,u,t),t,t).
\ef

This yields a corresponding map
$\psi_{n+1}: \mathbb{A}^3_{s,u,t} \rightarrow \mathbb{A}^3_{s,u,t}$,
which satisfies
\beq\label{t4}
   \psi_{n+1}(s,u,t) = \psi^{(n+1)}(s,u,t) = \psi(s_n,u,t) = (p(s_n,t,t),t,t) = (s_{n+1},t,t).
\ef

\begin{remark}\label{-id}
If $n\geq 1$ and $\tr(y) \ne 0$ then $e_n(x,y)\ne -id$, since $\tr((
-id)y)=-\tr(y) \ne \tr(y)$ in contradiction to \eqref{t1}.
\end{remark}

Define $H=\{(x,y)\in \tG\times \tG | \tr(xy)=\tr(y)\}$ and
$A = \{(s,u,t) \in \mathbb{A}^3_{s,u,t} | \ u=t\} \cong \mathbb{A}^2_{s,t}$,
then $\pi(H) \subseteq A$.
Equation \eqref{t4} now shows that in order to find the image of
$\psi_n: \mathbb{A}^3_{s,u,t} \rightarrow \mathbb{A}^3_{s,u,t}$,
one may consider its restriction $\mu^{(n)}: \mathbb{A}^2_{s,t}\rightarrow\mathbb{A}^2_{s,t}$,
where $\mu(s,t)=(s^2-st^2+2t^2 - 2,t)$.

\begin{defn}\label{morphisms}
Let us introduce the following morphisms:

\begin{itemize}
\item $\vp_n:\tG\times \tG \rightarrow \tG\times \tG, \ \
\vp_n(x,y) = (e_{n+1}(x,y),y), \ \vp_n(x,y)=\vp_0^{(n+1)}(x,y) ;$

\item $\th: \tG\times \tG \rightarrow \tG, \ \ \th (x,y)=x  ;$

\item $\tau: \tG \to \BA^1_{s}, \ \ \tau (x)=\tr(x)  ;$

\item $\lambda_1: \mathbb{A}^2_{s,t} \to \BA^1_{s}, \ \ \lambda_1 (s,t)=s  ;$

\item $\lambda_2: \mathbb{A}^3_{s,u,t}\to \mathbb{A}^2_{s,t}, \ \ \lambda_2(s,u,t)= (s,t);$

\item $\mu: \mathbb{A}^2_{s,t}\rightarrow\mathbb{A}^2_{s,t}, \ \
         \mu(s,t)=( s^2-st^2+2t^2 - 2, t) ;$

\item $\mu_n=\mu^{(n)};$

\item $\rho_n:\mathbb{A}^2_{s,t} \rightarrow \mathbb{A}^1_{s}, \ \ \rho_n=\lambda_1\circ\mu_n ;$

\end{itemize}
\end {defn}

These morphisms determine the following commutative diagram:
\beq\label{d1}
\xymatrix{
{\tG \times \tG} \ar[d]^{\pi} \ar[r]^-{\vp_0}
& H \ar[d]^{\pi} \ar[r]^{\vp_0^{(n)}}
& H \ar[d]^{\pi} \ar[r]^{\theta} &
{\tG} \ar[dd]^{\tau} \\
{\BA^3_{s,u,t}} \ar[r]^{\psi}
& A \ar[d]^{\lambda_2} \ar[r]^{\psi^{(n)}}
& A \ar[d]^{\lambda_2} \\
& {\BA^2_{s,t}} \ar[r]^{\mu_n} & {\BA^2_{s,t}} \ar[r]^{\lambda_1} & {\BA^1_{s}}
}
\ef

\begin{remark}
$\th\circ\vp_n(x,y)=e_{n+1}(x,y)$ and $\psi_{n+1}(s,u,t) = \bigl(\rho_n\circ\lambda_2(\psi(s,u,t)),t,t \bigr)$.
\end{remark}

Equation~\eqref{psi} shows that the morphism $\tG^2 \to \tG^2$
defined as $(x,y) \mapsto (e_n(x,y),y)$ is not dominant, since the
trace map $\psi$ of the first Engel word $e_1(x,y)=[x,y]$ maps the
three-dimensional affine space $\BA^3$ into a plane $A= \{u=t\}.$
One can consider the trace maps of the following Engel words
$e_{n+1}$ as the compositions of this map $\psi$ with the
endomorphism $\mu_n$ of $A$.

First, in \propref{image}, we find the image $\psi(\BA^3)\subset A$ and
then in \propref{if} we establish the connection between the image  of $\mu_n$
and the range of the corresponding Engel word $e_{n+1}$.
In the next section we shall study the properties of $\mu_n$.

\begin{prop}\label{image}
The image $\Psi_q=\psi(\BA^3_{s,u,t} (\BF_{q}))$ is equal to:
\begin{enumerate}
\item $A(\BF_{q})$, if $q$ is even;
\item $A(\BF_{q})\setminus Z_q\subset A(\BF_{q})$, if $q$ is odd,
where
$$Z_q=\{(s,t,t)\in A  | \ t^2=4 \text{ and} \ s-2  \ \text{is
not a square in } \BF_{q}\}.$$
\end{enumerate}
\end{prop}

\begin{proof}
A point $(s,t,t)\in\Psi_q$ if $C_{s,t}(\BF_{q})\ne\emptyset$, where
$$C_{s,t}=\{(s',u,t)|\ p(s',u,t)=s\}.$$

Now, $$ p(s',u,t)-s=s'^2+u^2+t^2-us't-2-s.$$

{\bf Case 1.} \emph{$q$ is even.} Then the equation
$$ p(s',u,t)-s=s'^2+u^2+t^2-us't-2-s=0$$
has an obvious solution $s'=0, u^2=t^2+s,$ since every number in
$\BF_q$ is a square.

{\bf Case 2.} \emph{$q\ge 3$ is odd.} Then
$$p(s',u,t)-s=s'^2+u^2+t^2-us't-2-s =
\bigl(s'-\fr{ut}{2}\bigr)^2-u^2\bigl(\frac{t^2-4}{4}\bigr)+t^2-2-s.$$
Thus, $C_{s,t}$ for a fixed $t$, is a smooth conic if $t^2-2-s\ne 0$
and $t^2\ne 4,$ with at most two points at infinity. If $t^2-2-s =0$
then $C_{s,t}$ is a union of two lines
$$\bigl\{\bigl(s'-\fr{ut}{2}\bigr)-\frac{u}{2}\sqrt{t^2-4}=0\bigr\}
  \cup \bigl\{\bigl(s'-\fr{ut}{2}\bigr)+\frac{u}{2}\sqrt{t^2-4}=0 \bigr\}$$
which have a point $(s'=0,u=0)$ defined over any field,
provided $t^2-4\ne 0.$

If $t^2-4= 0,$ then the equation $$p(s',u,t)-s=
\bigl(s'-\fr{ut}{2}\bigr)^2+2-s=0$$ has a solution if and only if
$s-2$ is a perfect square.
\end{proof}

\begin{defn}\label{def.sets}
Let us define the following sets:
\begin{itemize}
\item $E_{n+1}=\th\circ\vp_n(\tG\times\tG) =
\{z \in \tG: \text{ there exist } (x,y) \in \tG \times \tG \text{ s.t. } e_{n+1}(x,y)=z \} ;$

\item $Y_q = \lambda_2(\Psi_q) ;$

\item $Y'_q = \lambda_2(\Psi_q) \setminus \{(s,t):t=0\} ;$

\item $T_{n}(\BF_q)=\rho_n(Y_q) ;$

\item $T'_{n}(\BF_q)=\rho_n(Y'_q) .$
\end{itemize}
\end{defn}

\begin{prop}\label{if} $  $

\begin{enumerate}\renewcommand{\theenumi}{\Alph{enumi}}
\item If $q>2$ is even and $a \in \BF_q$, then the following two statements are equivalent:
\begin{enumerate}\renewcommand{\theenumii}{\it \roman{enumii}}
\item 
$a \in T_{n}(\BF_q)= \rho_n(\lambda_2(A(\BF_q))); $
\item 
Any element $z\in \tG$ with $\tr(z)=a$ belongs to $E_{n+1}$.
\end{enumerate}

\item If $q>3$ is odd and $a \in \BF_q$, $a \ne -2$, then the following two statements are equivalent:
\begin{enumerate}\renewcommand{\theenumii}{\it \roman{enumii}}
\item 
$a \in T_{n}(\BF_q)= \rho_n(\lambda_2(\Psi_q)); $
\item 
Any element $z\in \tG$ with $\tr(z)=a$ belongs to $E_{n+1}$.
\end{enumerate}

\item If $q>3$ is odd and $-2\in T'_n(\BF_q)$ then every element $z\in \tG$, $z \ne -id$,
with $\tr(z)=-2$ belongs to $E_{n+1}$.
\end{enumerate}
\end{prop}

\begin{proof}
It is obvious that if $z=e_{n+1}(x,y),$ then $a=\tr(z)=\rho_n \circ \lambda_2(\psi(\tr(x),\tr(xy),\tr(y)).$
Thus we need to prove the implications $(i) \Rightarrow (ii)$.

Assume that $a=\rho_{n}(s,t)$ for some $(s,t)\in Y_q=\lambda_2(\Psi_q)$.
Since $\psi$ is surjective onto $\Psi_q,$ there exists a
point $(s',u,t)\in \BA^3(\BF_q)$
such that $(s,t,t)=\psi(s',u,t)$.
Since the morphism $\pi$ is surjective for any field,
one can find $(x',y')\in \tG\times \tG$
such that $\pi(x',y')=(s',u,t)$. Let $v=e_{n+1}(x',y')$,
then $\tr(v)=a$ (see digram \eqref{d1}).

\vskip 0.2 cm

{\bf Case 1.} \emph{Either $q$ is even and $a \ne 0$, or $q$ is odd and $a\ne \pm 2$}.

In this case, $a=\tr(z)= \tr(v)$ implies that $v$ is conjugate to $z,$
i.e. $z=gvg^{-1}$ for some $g\in \tG$.
Therefore $e_{n+1}(gx'g^{-1},gy'g^{-1})=gvg^{-1}=z$,
and so one can take $x= gx'g^{-1},y=gy'g^{-1}$.

\vskip 0.2 cm

{\bf Case 2.} \emph{Either $q$ is even and $a=0$, or $q$ is odd and $a=2$}.

Observe that $2$ always belongs to $T_{n}(\BF_q)$ since $2-2=0$ is a
perfect square and $(2,t)$ is a fixed point of $\mu_n$.

It suffices to prove that all matrices
$w=\begin{pmatrix}1&c\\0&1\end{pmatrix}$, $c \in \BF_q$, are in the
image $E_n$. Since
$$e_n\left(\begin{pmatrix}1&b\\0&1\end{pmatrix},\begin{pmatrix}a&0\\0&\frac{1}{a}\end{pmatrix}\right)=
\begin{pmatrix}1&b(1-a^2)^n\\0&1\end{pmatrix},$$
one can take some $0 \ne a \in \BF_q$ with $a^2 \ne 1$ and $b=\frac{c}{(1-a^2)^n}$.

{\bf Case 3.} \emph{$q$ is odd and $a=-2$}.

If $-2 \in T'_n(\BF_q)$ then $v\ne -id$ by \remarkref{-id}. Choose
$\a\in \BF_{q^2} \setminus \BF_q$ such that $\a^2\in \BF_q.$ Let
$$m=\begin{pmatrix}\a&0\\0&\fr{1}{\a}\end{pmatrix}.$$
Then $mvm^{-1} \in \tG$ and moreover, either $v$ or $mvm^{-1}$ is conjugate to $z$ in $\tG$.

If $v$ is conjugate to $z$, then we proceed as in {\bf Case 1}.

If $mvm^{-1}$ is conjugate to $z$, then we consider the pair
$\bigl(x''=mx'm^{-1}, y''=my'm^{-1}\bigr) \in \tG \times \tG.$ We
have $mvm^{-1}=e_{n+1}(x'',y''),$ and we may continue as in {\bf
Case 1}.
\end{proof}

\begin{corr}
If $a\in \BF_q ,\ a\ne -2$ belongs to the image
$\rho_n(\BA^2_{s,t})(\BF_q),$ then any element $z\in \tG$ with
$\tr(z)=a$ belongs to $E_{n}.$
\end{corr}
\begin{proof}
Indeed, $\rho_n(\BA^2_{s,t})\subseteq
\rho_{n-1}(\lambda_2(\Psi_q)),$ because
$\psi(s,t,t)=(\rho(s,t),t,t).$
\end{proof}

\begin{defn}
When $q$ is odd, the point $(s,t)\in \BA^2_{s,t}$ is called an
\emph{exceptional} point if either $t=0$ or $t^2=4.$ The set of all
exceptional points is denoted by $\Upsilon$.
\end{defn}

\begin{corr}\label{notex}
If either $q$ is even and $a\in \rho_n(\BA^2_{s,t})$, or
$q$ is odd and $a\in \rho_n(\BA^2_{s,t}\setminus \Upsilon)$, then
any element $z\in \tG = \tG(\BF_q)$ with $\tr(z)=a$ belongs to
$E_{n+1},$ i.e there exist  $(x,y)\in \tG\times \tG$
such that $z=e_{n+1}(x,y)$.
\end{corr}

\begin{corr}\label{plmin}
If $q$ is odd and $T_n(\BF_q)$ contains either $a$ or $-a$
for every $a\in \BF_q$ then the Engel word map $e_{n+1}$ is surjective on $\PSL(2,q)$.
\end{corr}

\begin{proof}
It follows from \propref{if}(B) and the fact that both elements
$z\in\SL(2,q)$ and $-z\in\SL(2,q)$ represent the same element of $\PSL(2,q)$.
\end{proof}

\section{Surjectivity and non-surjectivity of Engel words over special fields}\label{special.fields}

The following examples show that the $n$-th Engel word map (for
$n\geq 3$) is not always almost surjective on $\SL(2,q)$ (in light
of \propref{if}). However, it is still conjectured that it is
surjective on $\PSL(2,q)$ (see \conref{shalev}).

\begin{example}\label{ex.no.sol}
In the following cases, computer experiments using \textsc{Magma}
show that there is no solution to $\rho_n=a$ in $\BF_q$.
\begin{itemize}
\item There is no solution in $\BF_{11}$ to $\rho_n=9$ for every $n \geq 2$.
\item There is no solution in $\BF_{13}$ to $\rho_n=4$ for every $n \geq 5$.
\item There is no solution in $\BF_{17}$ to $\rho_n=10$ for every $n \geq 2$,
to $\rho_n=4$ for every $n \geq 4$, and to $\rho_n=5$ for every $n \geq 5$.
\item There is no solution in $\BF_{23}$ to $\rho_n=16$ for every $n \geq 2$.
\item There is no solution in $\BF_{53}$ to $\rho_n=31$ for every $n \geq 8$.
\item There is no solution in $\BF_{67}$ to $\rho_n=4$ for every $n \geq 10$.
\end{itemize}
\end{example}

\begin{remark}
In fact, it is sufficient to check any of the above examples for all
integers $n \leq q$, since  for every $(s,t) \in \BF_q^2$ there exist some
$N\leq q$ such that $\mu_{N}(s,t)$ is a periodic point of $\mu.$
\end{remark}

Following some further extensive computer experiments using
\textsc{Magma}, in which we checked all $q<600$ and $n<50$, we
moreover suggest these conjectures (see also \propref{sne1} below).

\begin{conj}\label{conj.SL}
For every finite field $\BF_q$, $a \in \BF_q$ and $n \in \BN$,
unless either $a=1$ and $\sqrt{2} \notin \BF_q$, or the triple
$(q,a,n)$ appears in one of the cases in Example~\ref{ex.no.sol},
one has that $\rho_n$ attains the value $a$.
\end{conj}

\begin{conj}\label{conj.PSL}
For every finite field $\BF_q$, $a \in \BF_q$ and $n \in \BN$,
either $a$ or $-a$  is in the image of $\rho_n.$
\end{conj}

Observe that if the first conjecture is true then so is the second.


\vskip 0.3cm

We continue by considering some special infinite families of finite fields.
We will mainly use the following properties of the maps $\mu_n$ and $\rho_n$.

\begin{property}\label{properties} ${}$
\begin{enumerate}
\item\label{p1}$\mu(1,t)=(t^2-1,t);$
\item\label{p2}$\mu(t^2-1,t)=(t^2-1,t);$
\item\label{p3}$\mu(2,t)=(2,t);$
\item\label{p32}$\mu(t^2-2,t)=(2,t);$
\item\label{p4} $\rho_n(s,0)=x^{2^n}+\fr{1}{x^{2^n}}$,  \  if $s=x+\fr{1}{x};$
\item\label{p5} $\rho_n(s,t)=(s-1)^{2^n}+1$, \ if $t^2=2;$
\item\label{p6} $\rho_n(s,t)=(s-2)^{2^n}+2$, \ if $t^2=4.$
\end{enumerate}
\end{property}

\begin{corr}\label{p7}
Let $t \in \BF_q.$ Then $t^2-1$ is in $T_n(\BF_q)$  for every $n$.
\end{corr}
\begin{proof}
Item (\ref{p2}) implies that the point $(t^2-1,t)$ is a fixed point of $\mu$.
Moreover, if $t^2=4$, then $(t^2-1)-2=1$ is always a square, and hence $t^2-1 \in \Psi_q$ for every $q$.
\end{proof}


We shall now explain why $-id$ cannot appear in the image of long
enough Engel words, motivating \defnref{almostsurj} of ``almost
surjectivity''.

\begin{prop}\label{prop.minid}
If $n\geq 1$ and $q\geq 7$ is an odd prime power, then there is a
solution $(x,y) \in\tG^2$ to the equation $e_{n+1}(x,y)=-id$ if and
only if there exists some $c\in\BF_{q^2}$ satisfying $c^{2^n}=-1$.
\end{prop}

\begin{proof}
Assume that $e_{n+1}(x,y)=-{id}.$  Then, by \remarkref{-id}, there
exists some $b\in\BF_q$ such that $\rho_{n}(b,0)=-2.$  According to
\prref{properties}(\ref{p4}),
$$\rho_{n}(b,0)=c^{2^n}+\fr{1}{c^{2^n}},$$
where $c\in \BF_{q^2}$ is defined by the equation $b=c+\fr{1}{c}$.
Thus, $$c^{2^n}+\fr{1}{c^{2^n}}=-2,$$ implying that
      $$(c^{2^{n-1}}+\fr{1}{c^{2^{n-1}}})^2=0,$$
and so
      $$c^{2^n}=-1.$$

On the other direction, assume that there exists some
$c\in\BF_{q^2}$ satisfying $c^{2^n}=-1$, let $b=c+\frac{1}{c},$ and
denote
$$A= \begin{pmatrix} c&0\\ 0&\frac{1}{c} \end{pmatrix}.$$

Consider the rational curve $C$ defined by $s^2+u^2=b+2.$ Note that
$b+2\ne 0$ since $c\ne -1.$ Thus, being a smooth rational curve,
$C(\BF_q)$ has at least $q-1$ points. If $q\geq 7,$ there are points
$(s,u)$ in $C(\BF_q)$ such that $s\ne\pm 2.$ Let $(s,u)$ be such a
point, and let $x_0,y_0\in \SL(2,q)$ be any pair of matrices such
that $\tr(x_0)=s, \tr(x_0y_0)=u, \tr(y_0)=0.$

We shall show that $e_{n+1}(x_0,y_0)=-id.$ Consider $x_0$ and $y_0$
as elements of $\tG_1=\SL(2, F_1),$ where $F_1$ is a quadratic
extension of $\BF_q$ such that $c\in F_1.$ Let $\pi_1 :\tG_1^2\to
\BA^3(F_1)$ be the trace projection:
$$\pi_1(x,y)=( \tr(x), \tr(xy), \tr(y) ).$$
Then any pair $(x_1,y_1)$ satisfying that $\pi_1(x_1,y_1)=(s,u,0)$
is conjugate to the pair $(x_0,y_0)$ in $\tG_1$, namely, there
exists $g\in \tG_1$ such that $x_1=gx_0g^{-1},\ y_1=gy_0g^{-1}.$

Hence, $e_{n+1}(x_0,y_0)$ is conjugate in $\tG_1$ to
$e_{n+1}(x_1,y_1).$

Take
$$x_1=\begin{pmatrix}\frac{sc}{c+1}&\frac{uc}{c+1}\\ \frac{-u}{c+1}&\frac{s}{c+1}\end{pmatrix},
\quad y_1=\begin{pmatrix} 0&1\\ -1&0\end{pmatrix}.$$ A direct
computation shows that
$$[x_1,y_1]=\begin{pmatrix}\frac{(u^2+s^2)c^2}{(c+1)^2}&0\\ 0&\frac{(u^2+s^2)}{(c+1)^2}\end{pmatrix}=A.$$

Let us now compute $e_n(A,y_1).$ Let
\begin{equation*}
   X(a)= \begin{pmatrix} a & 0 \\0 & \frac{1}{a}\end{pmatrix}
\end{equation*}

Then \begin{equation*}
[X(a),y_1]=\begin{pmatrix} a^2 & 0
\\0 & \frac{1}{a^2}\end{pmatrix} ,
\end{equation*}
and so,
\begin{equation*}
e_n(X(a),y_1)=\begin{pmatrix} a^{2^n} & 0 \\0 &
\frac{1}{a^{2^n}}\end{pmatrix}.
\end{equation*}

Since $A=X(c),$ then 
$e_n(A,y_1)=-id.$ In addition, $A=e_1(x_1,y_1)$, and hence $e_{n+1}(x_1,y_1)=-id $.
But then $e_{n+1}(x_0,y_0)$ is
conjugate to $-id,$ and therefore $e_{n+1}(x_0,y_0)=-id$ as well.
\end{proof}

\begin{remark} \label{comm.min.id}
If $n>q$ then the equation $c^{2^n}=-1$ has no solution in
$\BF_{q^2}$, and hence $e_n(x,y)\ne -{id}$ for every $x,y\in
\SL(2,q)$.

However, $-id$ can be written as a commutator of two matrices in
$\SL(2,q)$, where $q$ is odd. Indeed, take $a,b \in \BF_q$
satisfying $a^2+b^2=-1$, then
$$
    \left[ \begin{pmatrix}0 & 1 \\ -1 & 0 \end{pmatrix},
    \begin{pmatrix}a & b \\ b & -a \end{pmatrix} \right] = -id.
$$
(See \cite{Th} and \cite{Bl} for a general result regarding central
elements in $\SL(n,q)$ and other quasisimple groups).
\end{remark}


We moreover show that there exists an infinite family of finite
fields $\BF_q$, for which the $n$-th Engel word map in $\SL(2,q)$ is
not even almost surjective for sufficiently large $n \geq n_0(q)$.

\begin{prop}\label{sne1}
Let $\BF_q$ be a finite field which does not contain $\sqrt{2}$. Then there exists
some integer $n_0=n_0(q)$ such that for every $n>n_0$, $\rho_n \neq 1$.
\end{prop}

\begin{proof}
Since the set of points $(s,t)\in\BA^2(\BF_q)$ is finite,
every point is either periodic or preperiodic for $\mu_n$. This means that
for every $(s,t)\in\BA^3(\BF_q)$ there are numbers $\tilde n(s,t)$ and $m(s,t)<\tilde n(s,t)$
 such that
$$\mu_{\tilde n(s,t)}(s,t)=\mu_{m(s,t)}(s,t).$$
For a point $(s,t)$ we define $n(s,t)$ as the minimum of all possible $\tilde n(s,t).$

Let $$n_0 = \max\{n(s,t): \ (s,t)\in\BA^2(\BF_q)\}.$$ Then every
$(s,t)\in R_{n_0}=\mu_{n_0}(\BA^2(\BF_q))$ is periodic and
$R_n=R_{n_0}$ for any $n\ge n_0$. In order to show that $\rho_n \neq
1$ it is sufficient to show that for any $t$, the point
$(1,t)\not\in R_{n_0}$, i.e. to show that it is not periodic.
Indeed, $\mu(1,t)=(t^2-1,t)$ which is a fixed point for any $t$.
Thus, for every $k>0$ we have $\mu_k(1,t)=(t^2-1,t)\ne(1,t)$, if
$t^2\ne 2$.
\end{proof}


On the other hand, we show that there are certain infinite families
of finite fields $\BF_q$ for which the $n$-th Engel word map in
$\PSL(2,q)$ is always surjective for every $n \in \BN$. The first
family consists of all finite fields of characteristic $2$, and the
second family contains infinitely many finite fields of odd
characteristic.


\begin{prop}\label{prop.even}
For every $n \geq 1$, the Engel word map $e_n$ is surjective on the group $\PSL(2,q)$ for $q=2^e$, $e>1$.
\end{prop}

\begin{proof}
In this case
\beq\label{qeven1}
    \mu(s,t) = (s^2 - st^2,t), \quad
    \rho(s,0) = s^2.
\ef

Thus $\rho(s,0)$ is an isomorphism of $\BA^1_{s}(\BF_q),$ as well as
any of its iterations $\rho_n(s,0)$. According to \propref{if}, this
implies the surjectivity of the $n$-th Engel word map on
$\PSL(2,q)=\SL(2,q)$.
\end{proof}


\begin{prop}\label{prop.sqrt}
For every $n \geq 1$, the Engel word map $e_{n+1}$ is surjective on the group $\PSL(2,q)$,
if $\sqrt{2}\in \BF_q$ and $\sqrt{-1}\not\in \BF_q$.
\end{prop}

\begin{proof}
By \corref{plmin}, we need to show that either $a\in T_n(\BF_q)$ or $-a\in T_n(\BF_q)$
for every $a\in \BF_q$.

In this case, the map $x\to x^2$ is a bijection on the subset of
perfect squares of $\BF_q$. It follows that if $a=b^2$ for some
$b\in \BF_q$, then for every $n$, there is some $b_n\in \BF_q$ such
that $a=b_n^{2^n}$. Moreover, for every $a\in \BF_q$ either $a=b^2$
for some $b\in \BF_q$ or $a=-b^2$ for some $b\in \BF_q$.

Assume that $z\in\PSL(2,q)$ and $z\ne e_{n+1}(x,y)$. Let $\tr(z)=a$.
Then, by Corollary~\ref{p7}, neither $a+1$ nor $-a+1$ is a square in
$\BF_q$. It follows that $a+1=-c^2$ and $-a+1= -b^2$ for some $b,c
\in \BF_q$. Hence, $a=b^2+1 =b_n^{2^n}+1=\rho_n(b_n+1,\sqrt{2})$
according to \prref{properties}(\ref{p5}), yielding $a\in
T_n(\BF_q)$.
\end{proof}

\section{Engel words in $ \SL(2,q)$ for sufficiently large $q$}\label{main}

In this section we prove \thmref{thm.engel} and show that the $n$-th
Engel word map $e_n$ is almost surjective on $\SL(2,q)$ if $q\geq
q_0(n)$ is sufficiently large. We moreover give an explicit estimate
for $q_0(n)$, which, unfortunately, is exponential in $n$.

Our proof has a geometrical flavor. Let us briefly describe it and
explain the geometric idea behind our calculations. Consider the
diagram \eqref{d1}. Instead of solving the equation $\tr
(e_{n+1}(x,y))=a,$ we look for points defined over a ground field
$\BF_q$ in the curve $\{\mu_n(s,t)=a\}.$ This is an affine curve. In
order to use the Weil inequality, we have to know that it has an
irreducible component defined over the ground field $\BF_q$, and we
need to estimate its genus and the number of punctures.

To this end we represent the curve as a tower of double covers of a
rational curve (see equation \eqref{s1}). The geometrical
interpretation of this procedure is an embedding of the curve into
an affine space of a higher dimension
$\BA^{n+2}_{z_1,...,z_n,\varkappa,t}$. Then we consider the closure
$X$ of this curve in the corresponding projective space
$\BP^{n+2}_{x_1:...:x_n:y:d:w}.$

It appears (see \lemref{infty}) that the intersection of $X$ with
the hyperplane at infinity consists of smooth points defined over
$\BF_q$ for any $q.$ Thus every irreducible component of $X$ is
defined over  $\BF_q$ as well. Indeed, assume that an irreducible
component (say, $X_i$) is defined over an extension of $\BF_q$ and
is not invariant under the action of the corresponding Galois group
$\Gamma$, then the $\Gamma-$invariant points would belong to the
intersection $X_i\cap \gamma (X_i), \gamma\in \Gamma,$ and
therefore, the points defined over $\BF_q$ would not be smooth.

The rest of the proof deals with the estimation of the genus and the
number of punctures.

\medskip

By \propref{prop.even} we may assume that $q$ is odd.
We continue to use the notation introduced in \defnref{morphisms}.

\begin{theorem}\label{qbig}
For every $n \in\BN$ there exists $q_0=q_0(n)$ such that
$\rho_n:\BA^2_{s,t}\setminus \Upsilon\to \BA^1_s$ is surjective for every field $\BF_{q}$ with $q\ge q_0$.
Moreover, if $n$ is a prime then there is an orbit of $\mu$ of length precisely $n$.
\end{theorem}

\begin{proof}
Together with the endomorphism $\mu:\mathbb{A}^2_{s,t} \rightarrow \mathbb{A}^2_{s,t}$
we may define the following endomorphism $m:\mathbb{A}^2_{z,\k} \rightarrow \mathbb{A}^2_{z,\k}$ by
\begin{equation}\label{r}
  m(z,\k)= (z(z-\k),\k).
\end{equation}

A direct computation shows that $\mu$ may be reduced to \eqref{r} by the
substitution $z=s-2, \ \k=t^2-4$.

Similarly to the morphisms $\lambda_1(s,t)=s$ and $\rho_n=\lambda_1\circ \mu^{(n)}$,
we may define the morphisms $l:\mathbb{A}^2_{z,\k} \rightarrow \mathbb{A}^1_{z}, \ \ l(z,\k)=z$ and
$r_n=l\circ  m^{(n)}$.

First, we note that $s=2$ is always in the image of $\rho_n$ (see
\propref{if}). Note also that $(s,t)=(-2,0)$ cannot be a periodic
point, since $\mu(-2,0)=(2,0)$, which is a fixed point.

Now, assume that some $a+2 \in \BF_{q}, a\ne 0$ is in the image of $\rho_n$.
This is equivalent to $a=r_n(z,\k)$ for some $z\in \BF_{q}$ and $\k=t^2-4, \ t\in \BF_{q}$.
The last statement implies that the following system of equations has a solution in $\BF_q$:

\begin{equation}\label{s1}\left\{
\begin{aligned}
    z_2 &= z_1(z_1-\k),\\
    & \ \vdots\\
    z_n &= z_{n-1}(z_{n-1}-\k),\\
    a &= z_n(z_n-\k),\\
    \k &= t^2-4.
\end{aligned}
\right.
\end{equation}

Similarly, the orbit of length $n$ is defined by the following system:
\begin{equation}\label{orb1}\left\{
\begin{aligned}
    z_2 &= z_1(z_1-\k),\\
    & \ \vdots \\
    z_n &= z_{n-1}(z_{n-1}-\k),\\
    z_1 &= z_n(z_n-\k),\\
    \k &= t^2-4.
\end{aligned}
\right.
\end{equation}

If $n$ is a prime, then  system \eqref{orb1} describes all
the points in an orbit either of exact length $n$ or of exact length $1$.
In the latter case, these points are $z_i=\k+1, i=1,\dots,n$ and $z_i=0,i=1,\dots,n$.

Consider the projective space $\BP^{n+2}(\ov {\BF}_q)$ with homogeneous coordinates
$\{x_1:\dots:x_n:y:d:w\}.$ Assume that $z_i=\frac{x_i}{w},i=1\dots,n,$
$\k=\frac{y}{w}, \ t=\frac{d}{w}$.
Then  system \eqref{s1} defines in $\BP^{n+2}$ a projective set

\begin{equation}\label{c1}X=\left\{
\begin{aligned}
    x_2w&=x_1(x_1-y),\\
    &\ \vdots\\
    x_nw&=x_{n-1}(x_{n-1}-y),\\
    aw^2&=x_n(x_n-y), \\
    yw&=d^2-4w^2.
\end{aligned}
\right.
\end{equation}

Similarly,  system \eqref{orb1} defines a projective set

\begin{equation}\label{c2}X_1=\left\{
\begin{aligned}
    x_2w&=x_1(x_1-y),\\
    &\ \vdots\\
    x_nw&=x_{n-1}(x_{n-1}-y),\\
    x_1w&=x_n(x_n-y),\\
    yw&=d^2-4w^2.
\end{aligned}
\right.
\end{equation}

\begin{lemma}\label{infty}
The intersections $S=X\cap\{w=0\}$ and $S_1=X_1\cap\{w=0\}$ consist
of $2^n$ smooth points with $w=0, d=0, y=1$ and $x_i=0$ or $1$ (for
$i=1,\dots,n$).
\end{lemma}

\begin{proof}
If there was a point in $X$ with $w=0, y=0$, then, according to
\eqref{c1} (respectively \eqref{c2}), $d$ and all $x_i$ would vanish
as well, which is impossible. Thus, $y\ne 0$ at the points of $S$
and $S_1$. But then \eqref{c1} (respectively \eqref{c2}) implies
that every $x_i$ is either $0$ or $y$ at the points of $S$
(respectively $S_1$).

It follows, in particular, that the sets $X$ and $X_1$ have no
components of dimension greater than $1$, since the intersection of
each such component with $\{w=0\}$ would be positive dimensional.

Let us compute the Jacobian matrices of these systems. We have for \eqref{c1}:
\begin{equation}\label{Jacobian1}
\begin{bmatrix}
\de_{d}&\de_{w}&\de_{y}&\de_{x_1}&\de_{x_2}& \dots& \de_{x_{n-1}} &\de_{x_n}\\
\\
 0    & -x_2   &  -x_1 &  2x_1-y &  -w     &\dots  &  0           &  0 \\
 0    & -x_3   &  -x_2 &    0    & 2x_2-y  &  -w   &\dots         &  0  \\
 \dots&\dots  &\dots   &\dots    &\dots    &\dots  &\dots        &\dots \\
 0    & -x_n  &  -x_{n-1}  &  0  & 0       &\dots  &  2x_n-y    &  -w   \\
 0    & -2aw  &  -x_n      &  0  & 0       &\dots  &  0         &  2x_n-y \\
 2d   & -8w-y  & -w        &   0  &   0      &\dots &     0      &   0
\end{bmatrix}
\end{equation}

and similarly for \eqref{c2}:

\begin{equation}\label{Jacobian2}
 \begin{bmatrix}
\de_{d}&\de_{w}&\de_{y}&\de_{x_1}&\de_{x_2}& \dots& \de_{x_{n-1}} &\de_{x_n}\\
\\
 0    & -x_2   &  -x_1 &  2x_1-y &  -w     &\dots  &  0           &  0 \\
 0    & -x_3   &  -x_2 &    0    & 2x_2-y  &  -w   &\dots         &  0  \\
 \dots&\dots  &\dots   &\dots    &\dots    &\dots  &\dots        &\dots \\
 0    & -x_n  &  -x_{n-1}  &  0  & 0       &\dots  &  2x_n-y    &  -w   \\
 0    & -x_1  &  -x_n      &  -w  & 0       &\dots  &  0         &  2x_n-y \\
 2d   & -8w-y  & -w        &   0  &   0      &\dots &     0      &   0   \end{bmatrix}
\end{equation}

Since at the points of $S$ and $S_1$ the ranks of these matrices are
$n+1$, every point is smooth.
\end{proof}

\begin{remark}
In particular, we have proved that the map $\rho_n$
is surjective over every algebraically closed field.
Indeed, every component of $X$  has dimension at least one, thus
no fiber is contained in the set $\{w=0\}$.
\end{remark}

Consider an irreducible component $A_i$ (over $\ov{\BF}_q$) of $X$ of degree $d_i$.
If it was not defined over $\BF_{q}$, then every point in $A_i$,
which is rational over $\BF_{q},$ would be singular. But, according to \lemref{infty},
$A_i$ has smooth points defined over $\BF_{q}$ (namely, $A_i\cap S$).
Thus, $A_i$ is defined over $\BF_{q}$.
Similarly, every irreducible component $B_i$ of $X_1$ is defined over $\BF_{q}$.

Let $\om:\BP^{n+2}\to \BP^{2}_{x_{1},d,w}$ be defined as
$\om(x_1:\dots:x_n:y:d:w)=(x_1:d:w).$ Then $\om$ induces a
birational map of every $A_i$ (respectively $B_i$) on its image
$R_i=\om(A_i)$ (respectively $U_i=\om(B_i)$), because of   \eqref{c1}
(respectively \eqref{c2}). Thus, $A_i$ is birational to the closure  $\ov Y^{(n)}$
in $\BP^{2}$ of an irreducible component of the set
 $$\tilde Y^{(n)}=\{r_n(z_1, t^2-4)=a\}\subset \BA_{z_1,t}^2,$$
which becomes $$Y^{(n)}=\{\rho_n(s,t)=a+2\}\subset \BA_{s,t}^2,$$
after the following change of a coordinates
 $s=2+z_1,$ ($x_1\to x_1+2w$).
Similarly, $B_i$ is birational to the closure $\ov Z^{(n)}$ in $\BP^{2}$ of an
irreducible component of the set
$$Z^{(n)}=\{\rho_n(s,t)=s\}\subset \BA_{s,t}^2.$$

The plane curves $R_i$ and $U_i$ are defined over the ground field
as the projections of $A_i$ and $B_i$ respectively. Let $d_n\le
3^{2^n}$ and $j_n\le 3^{2^n}$ be the degrees of $\ov Y^{(n)}$ and
$\ov Z^{(n)}$ respectively.

\vskip 0.2 cm

For the number $N(q)$ of points over the field $\BF_{q}$ in an
irreducible curve $C$ of degree $d$ in $\BP^{2},$ we use the
following analogue of the Weil inequality (see \cite{AP},
\cite[Corollary 7.4]{GL} and \cite [Corollary 2]{LY}),

\begin{equation}\label{weil2}
|C(\BF_{q})-(q+1)|\le (d-1)(d-2) \sqrt{q}.
\end{equation}

Hence, we obtain
\begin{equation}\label{weil3}
|R_i(\BF_{q})|\ge q+1- d_n^2 \sqrt{q} ,\\
\end{equation}
and
\begin{equation}\label{weil4}
|U_i(\BF_{q})|\ge q+1- j_n^2 \sqrt{q}.
\end{equation}

Now we need to check how many of these points can be exceptional or
at infinity. All these points are the intersection points with $4$
lines: $d=0,$  $d=\pm 2w,$ $w=0.$ By the B{\'e}zout's Theorem there
are at most $4d_n$ (respectively, $4j_n$) such points.

For any $q\ge 2d_n^4$ we have
$$q+1-d_n^2  \sqrt{q}\ge 2d_n^4+1-d_n^4\sqrt{2}  = d_n^4(2-\sqrt{2})+1>4d_n.$$

Similarly, for $q\ge 2j_n^4$,

$$q+1-j_n^2  \sqrt{q}\ge 2j_n^4+1-j_n^4\sqrt{2}  = j_n^4(2-\sqrt{2})+1>4j_n.$$

Thus, if $q \geq \max\{2d_n^4, 2j_n^4\}$ then $(R_i\setminus
\Upsilon)(\BF_{q})\ne\emptyset$ and $(U_i\setminus
\Upsilon)(\BF_{q})\ne\emptyset$, which completes the proof of
\thmref{qbig}.
\end{proof}

\begin{corr}\label{gr}
The map $e_n: \tG \times \tG \rightarrow \tG$ is almost surjective
if $\tG=\SL(2,q)$ and $q\geq q_0(n)$ is big enough.
\end{corr}

\begin{proof}
According to \corref{notex}, the almost surjectivity of $e_{n+1}$ on
$\SL(2,q)$ follows from the surjectivity of $\rho_n$ on
$\BA^2_{s,t}\setminus \Upsilon$, which was proven in \thmref{qbig}
for any $q\geq q_0(n)$.
\end{proof}

In order to  make the estimation for $q_0(n)$ more precise a
detailed study of  system ~\eqref{c1} is needed.

\begin{prop}\label{genus}
The curve $X$ defined in \eqref{c1} is irreducible provided $a\ne
0$. Let $\tilde \nu :\tilde X\to X $ be the normalization of $X.$
Then the genus $g(\tilde X)\leq 2^{n}(n-1)+1$ and $\tilde
\nu^{-1}(S) $ contains $2^n$ points.
\end{prop}
\begin{proof}

We will work over an algebraic closure of a ground field.
For $k=1,\dots,n$, we denote by $C_k$ a curve defined in $\BP^{n-k+2}$ by
\begin{equation}\label{ck}C_k=\left\{
\begin{aligned}
    x_{k+1}w &= x_k(x_k-y),\\
    &\ \vdots\\
    x_nw &= x_{n-1}(x_{n-1}-y),\\
    aw^2 &= x_n(x_n-y).
\end{aligned}\right.
\end{equation}

\begin{lemma}\label{genusc1}
If $a\ne 0$ and $q$ is odd, then the system \eqref{ck}  for $k=1$ defines in $\BP^{n+1}$
a smooth irreducible projective curve $C_1$ of genus $g(C_1)\leq 2^{n-1}(n-2)+1$.
\end{lemma}

\begin{proof}
Let $g_k$ denote the genus $g(C_k)$ (if $C_k$ is irreducible).

We shall prove by induction on $r=n-k$ that all curves $C_k$ are irreducible and moreover
$$g_k \le 2^{n-k}(n-k-1)+1.$$

\vskip 0.3 cm

{\bf Step 1.}
It is obvious that $C_n$ ($r=0$) is an irreducible conic in $\BP^2$ and that $g_n=0$.
At a point $(\a:\be:1)\in C_n$ we may use the affine coordinates
$z_i=\fr{x_i}{w},\varkappa=\fr{y}{w}$.
A local parameter on $C_n$ at this point may be taken as   $z_n-\a$, since
$$\varkappa -\be =(z_n-\a) \left(1+\fr{\a-\be}{z_n}\right)$$
(see, for example, \cite[I, Chapter 2, \S 1.6]{DS} for a definition
of a local parameter).

\vskip 0.3 cm

{\bf The induction step.}
Assume that for $r=n-k$ the assertion is valid, namely:
\begin{itemize}
\item The curve $C_k$ is smooth and irreducible;
\item $z_k-\a_k $ is a local parameter at every point $(\a_k : \dots : \a_n:\be:1)\in C_k$ ($w\ne 0$);
\item $g_k \le 2^{n-k}(n-k-1)+1.$
\end{itemize}

The curve $C_{k-1}$ is a double cover of $C_k$,
since to the equations defining $C_k$ one equation for the new variable $x_{k-1}$ is added:
$$x_{k}w=x_{k-1}(x_{k-1}-y).$$
Thus,
$$x_{k-1}=\fr{y}{2}\pm\sqrt{\fr{y^2}{4}+wx_k}.$$

It follows that the double points are:
$$x_{k-1}=\fr{y}{2}, \ x_k=-\fr{y^2}{4w}.$$

Note that $w\ne 0$ at a ramification point.
Indeed, if $w=0$ and $\sqrt{\fr{y^2}{4}+wx_k}=0$ then $y=0$,
which is impossible in light of  \lemref{infty}.
Thus we may take $w=1$.

Hence in affine coordinates, at the double point
$ ( \fr{\be}{2}:-\fr{\be^2}{4}:\dots :\a_n: \be:1)\in C_{k-1}$,
we have
\begin{itemize}
\item
$\fr{\be^2}{4}+z_k$ is
a local parameter on $C_k$ by the induction hypothesis;
\item
$(z_{k-1}- \fr{\be}{2})^2={\fr{\be^2}{4}+z_k}.$
\end{itemize}

It follows that:
\begin{itemize}
\item This point is a ramification point indeed;
\item $z_{k-1}- \fr{\be}{2}$ is a local parameter on $C_{k-1}$ at this point;
\item $C_{k-1}$ is smooth at this point.
\end{itemize}

Outside the ramification points, the projection $C_{k-1}\to C_k$
is \emph{\'{e}tale}. At infinity all the points are smooth, see \lemref{infty}.
Therefore, since $C_k$ is smooth and irreducible by the induction assumption,
then $C_{k-1}$ is smooth and irreducible as well.

\vskip 0.2cm

Let us compute the number of ramification points.
We have:
\begin{equation}\label{points}
\begin{aligned}
    x_{k-1} =\fr{y}{2}, \
    x_k =-\fr{y^2}{4}, \
    x_{k+1} =-\fr{y^2}{4}(-\fr{y^2}{4}-y), \
\dots \ ,  \\
    x_{k+s} =p_s(y), \
\dots \ ,
    x_{n}=p_{n-k}(y), \
    a = p_{n-k+1}(y),
\end{aligned}
\end{equation}
where $p_s(y)$ is a polynomial in $y$ and $\deg p_s(y)=2^{s+1}$.

Hence, the last equation has $l \le 2^{n-k+2}$  distinct roots.

\vskip 0.2cm

By the Hurwitz formula (see e.g. \cite[I, Chapter 2, \S 2.9]{DS})
and the induction estimate for $g_k$ we obtain:
\begin{equation*}
\begin{aligned}
g_{k-1} = 2g_k-1+\fr{l}{2} &\le 2(2^{n-k}(n-k-1)+1)-1+2^{n-k+1} = \\
        & 2^{n-(k-1)}(n-k)-2 \cdot 2^{n-k} + 2^{n-k+1}+2-1=2^{n-(k-1)}(n-k)+1.
\end{aligned}
\end{equation*}
This completes the induction. Thus, $g_1 \le 2^{n-1}(n-2)+1$.
\end{proof}

Now, the curve $X$ is obtained from $C_1$ by adding one more equation:
$$wy=d^2-4w^2,$$
(this is the last equation of  system \eqref{c1}).
It follows that $X$ is a double cover of $C_1$ with double points at $w=0$ or $y=-4w$.
At every such point $y\ne 0.$  Moreover, $X$ is smooth at every point of $S$
(see \lemref{infty}), hence every point of $S$ is a ramification point.
Thus, $X$ is irreducible. Moreover, $\tilde \nu$ is
one-to-one at these points.

Any other double point is either a ramification point or a double self-intersection.
Since $d^2=wy+4w^2$, these are points with $y=-4w.$ Similarly to ~\eqref{points},
there can be at most $2^n$ such points at $X.$

From the Hurwitz formula we obtain:
$$g(\tilde X)\le 2g(C_1)-1+2^n=2(1+2^{n-1}(n-2))-1+2^n=2^n(n-1)+1.$$

This completes the proof of \propref{genus}.
\end{proof}

\begin{remark}
The more detailed analysis of the curve $X$ shows that it is not smooth only if $a=-4.$
If $a\ne -4$ the normalization is not needed.
\end{remark}

\begin{corr}\label{less}
For any $n>2$, the map $e_{n+1}: \tG \times \tG \rightarrow \tG$ is
almost surjective if $\tG=\SL(2,q)$ and $q>2^{2n+3}(n-1)^2$.
\end{corr}

\begin{proof}
We want to prove that any number $a\in \BF_q$ is attained by $r_n.$
Since the normalization $\tilde X$ of $X$ is defined over the ground
field (see \cite[Chapter 1, \S 6.4 and \S 7]{Sa}), every point
$\tilde x\in \tilde X(\BF_q)$ provides a point $\tilde{\nu}(\tilde
x)\in X(\BF_q).$ In order to exclude the exceptional points, we
should take away from $X$ the following points:
\begin {itemize}
\item $2^n$ points of $S;$
\item $2^{n+1}$ points  with $y =0, d=\pm 2w;$
\item $2^{n}$ points with $y =-4w,d=0.$
\end{itemize}
Since (for $a=-2$) the points with  $y =-4w,d=0$ may be
selfintersections, we should count them twice. Thus
we need that $|\tilde X(\BF_q)|>5\cdot 2^{n}.$

We shall use the Weil inequality (see \cite{AP}) once more.
For a field $\BF_q$ we need that
\begin{equation}\label{Wi}
q+1-2g\sqrt{q}-\delta> 0,
\end{equation}
where, by \propref{genus}, $g\leq 2^{n}(n-1)+1,$ and $\delta =5\cdot 2^{n}$.
Take $q\ge 2^{2n+3}(n-1)^2.$

Then
\begin{equation}
\begin{aligned}
q+1-2g\sqrt{q}-\delta &\ge 2^{2n+3}(n-1)^2+1-2(2^{n}(n-1)+1)2^{n+1}(n-1)\sqrt{2}-5\cdot 2^{n} \\
&\ge 2^n \bigl( 2^{n+3}(n-1)^2-2^{n+2}(n-1)^2\sqrt{2}-4\sqrt{2}(n-1)-5 \bigr)>0,
\end{aligned}
\end{equation}
for any $n>2.$
\end{proof}

\section {Short Engel words in $\PSL(2,q)$} \label{small.n}

In this section we prove \corref{cor.short} and show that for any $n
\leq 4$ the $n$-th Engel word map is surjective for all groups
$\PSL(2,q)$. From \corref{plmin} it follows that in order to prove
that the map $e_{n+1}:G \times G \to G$ is surjective, one should
check that for every $a\in \BF_q$ either $a$ or $-a$ belongs to the
image $T_n$ of $\rho_n$. For a fixed $n$ and $q$ big enough it
follows from \thmref{qbig}, and so for small values of $q$ it may be
verified by computer. Indeed, we have done the following
calculations for small values of $n$ using the \textsc{Magma}
computer program.

\subsection*{Case $e_1=[x,y]$:}\label{1}
In this case, the surjectivity follows from \propref{image},
\propref{if} and \remarkref{comm.min.id}. This provides an
alternative proof to the well-known fact that any element in the
group $\SL(2,q)$ (and in the group $\PSL(2,q)$), when $q>3$, is a
commutator (see \cite{Th}).

\subsection*{Case $e_2=[x,y,y]$:}\label{2}
We need to prove that the map $\rho_1$ is surjective.
Indeed, the equation
$$\rho_1(s,t)-a=s^2-st^2+2t^2-2-a=0$$
defines a smooth curve of genus $1$ with two punctures if $a^2\ne
4$. Thus if $q>7$ it has a point over $\BF_q$. The case $a=2$ was
dealt with in \propref{if}. The cases $q=5,7$ can easily be checked
by a computer. Therefore, $e_2$ is surjective on $\SL(2,q) \setminus
\{-id\}$, and hence on $\PSL(2,q)$, for any $q>3$.

\subsection*{Case $e_3=[x,y,y,y]$:}\label{3}
Recall that by Example \ref{ex.no.sol}, $e_3$ is no longer
surjective on $\SL(2,q)$. In this case, the curve $\rho_2(s,t)-a$
has genus $2^2+1=5$ and it has at most $20$ punctures at $\infty,$
$t^2=4$ and $t=0.$ Thus the techniques of \secref{main} may be
applied for any $q$ which satisfies that
$$q+1-10\sqrt{q}-20>0,$$
namely, for any $q\ge 137$. For $q<137$ the surjectivity on
$\PSL(2,q)$ was checked by a computer.

\subsection*{Case $e_4=[x,y,y,y,y]$:}\label{4}
In this case $g=17$, and the computations were done for all
$q \le 1240$.

\section{Equidistribution of the Engel words in $\PSL(2,q)$} \label{sect.equi}

In this section we prove \thmref{thm.equi} by showing first that the
$n$-th Engel word map is almost equidistributed for the family of
groups $\SL(2,q)$, where $q$ is odd, and then explaining how this
implies that the $n$-th Engel word map is almost equidistributed
(and hence also almost measure preserving) for the family of groups
$\PSL(2,q)$, where $q$ is odd.

More precisely, for $g \in \tG=\SL(2,q)$, let
\[
   E_n(g) = \{(x,y) \in \tG \times \tG: e_n(x,y)=g \},
\]
then by \defnref{def.equi} we need to prove the following:

\begin{prop}\label{equi.q}
If $q$ is an odd prime power, then the group $\tG=\SL(2,q)$ contains
a subset $S= S_{\tG} \subseteq \tG$ with the following properties:
\begin{enumerate}\renewcommand{\theenumi}{\it \roman{enumi}}
\item $|S| = |\tG|(1-\ep)$;
\item $|\tG|(1-\ep) \leq |E_n(g)| \leq |\tG|(1+\ep)$ uniformly for all $g \in S$;
\end{enumerate}
where $\ep \to 0$ as $q \to \infty$.
\end{prop}

For the commutator word $e_1=[x,y]$, \thmref{thm.equi} has already
been proved in \cite[Proposition 5.1]{GS}. Hence we may assume that
$n>1$. Following \secref{main} we continue to assume that $q$ is
odd. We maintain the notation of \defnref{morphisms}.

\begin{proof}[Proof of \propref{equi.q}]

Consider the following commutative diagram of morphisms:

\beq\label{d2}
\xymatrix{
{\tG \times \tG} \ar[d]_{\gamma} \ar[dr]^{\alpha} \ar[r]^-{\pi}
& {\BA^3_{s,u,t}} \ar[r]^{p'}
& {\BA^2_{s,t}} \ar[dl]_{\rho_n} \ar[d]^{\mu_n} \\
{\tG} \ar[r]^{\tau} & {\BA^1_{s}} & {\BA^2_{s,t}} \ar[l]_{\lambda_1}
}
\ef

Here, $\g=\th\circ\vp_n=e_{n+1}$, $p'(s,u,t)=(s^2+t^2+u^2-ust-2,t)$,
and $\a$ is a composition of the corresponding morphisms in the
diagram.

We denote $f^{-1}(a)=f^{-1}(a)(\BF_q)$.

Let $a\in \BF_q, \ a\ne \pm 2$. Then $\a^{-1}(a)$ is a union of the
fibers $\G_z=\g^{-1}(z)$, where $z\in\tG$ is an element with
$\tr(z)=a$. Since  $a\ne\pm 2$, any $\G_z$ may be obtained from any
other $\G_{z'}$ (with $\tr(z')=a$) by conjugation, and so
$|E_n(z)|=|E_n(z')|$. Hence, \beq\label{gamma} |\g^{-1}(z)|=\fr{|
\a^{-1}(a)|}{|\tau^{-1}(a)|} . \ef

Recall that $|\SL(2,q)| = q^3-q$.

Take the set $S=S_{\tilde G}=\{z \in \tilde G: \tr(z) \ne \pm 2\}$,
then
$$|S|=q^3-2q^2-q=q^3\bigl(1-O(1/q)\bigr),$$
satisfying condition {\it(i)}.

In order to prove condition {\it(ii)} it is enough to show that for
any $z \in \tG$ with $\tr(z) = a \neq \pm 2$,
\[
    |\g^{-1}(z)| = q^3(1+\tilde \ep),
\]
where $\tilde\ep \to 0$ as $q \to \infty$.

\vskip 0.2cm

It is well-known that \beq\label{tau}
|\tau^{-1}(a)|=q^2(1+\ep_1(q)),\ef where $|\ep_1(q)|\leq\fr{1}{q}$
(see, for example \cite{Do}).

On the other hand, $\a=\rho_n \circ p' \circ \pi$. Let us estimate
$|\a^{-1}(a)|$.

\begin{lemma}\label{schet}
Let $\tilde p=p'\circ\pi.$ Then there are constants $M_1$ and $M_2$
such that for every $(s,t)\in\BA^2_{s,t}$, $s\ne 2$,
\begin{enumerate}
\item  If $t^2\ne 4 $ and $s\ne t^2-2$  then $|\tilde p^{-1}(s,t)|=q^4(1+\ep_2),$
where $|\ep_2|\leq\fr{M_1}{q};$
\item If $t^2= 4 $ then   $|\tilde p^{-1}(s,t)|\leq M_2q^4.$ \end{enumerate}
\end{lemma}

\begin{proof}  We use the notation of \propref{image}.

{\bf(1)}.  Assume that $t^2\ne 4 $ and $s\ne t^2-2.$ According to
Case 2 of  \propref{image}, \beq\label{i1} |
p'^{-1}(s,t)|=|C_{s,t}(\BF_q)|=q\pm 1. \ef

For a point  $(s',u,t)\in C_{s,t}(\BF_q)$  we shall now compute
 $|\pi^{-1}(s',u,t)|.$

We fix a matrix
 $$y_t=\begin{pmatrix} t & 1 \\-1 & 0\end{pmatrix}$$
with $t^2\ne 4$. Direct computation shows that $(x,y_t)\in
\pi^{-1}(s',u,t)$ if

\begin{equation}\label{f4}
x=\begin{pmatrix} a & b \\ u+b-at & s'-a \end{pmatrix}
\end{equation}
satisfies that
$$\dl^2-\om^2\st^2=p(s',u,t)-2,$$
where
\begin{equation}\label{f5}\begin{aligned}
\om^2 &= t^2-4, \\
\st &=a-\fc{bt}{2}-\fc{s'}{2},\\
\dl &=-u+\fc{s't}{2}+\fc{\om^2b}{2}.\end{aligned}
\end{equation}

Thus, we have a conic once more, and the number of such $x$ is
therefore $q\pm 1$.  Together with \eqref{tau} and \eqref{i1} one
has that
$$|\tilde p^{-1}(s,t)|=(q\pm 1)(q\pm
1)q^2(1+\ep_1(q))=q^4(1+\ep_2(q)),$$ where
$$|\ep_2|\leq\fr{2}{q}+\ep_1(q)+O\bigl(\fr{1}{q^2}\bigr)\leq
\fr{4}{q}.$$

{\bf(2.1)}. Assume that $t=2.$  Then (see Case 2 of
\propref{image}), \beq\label{i2} |C_{s,t}(\BF_q)|\leq 2q, \ef where
$s-2=v^2,$ and $s'-u=\pm v$ for some $v\in\BF_q$ and any
$(s',u,t)\in C_{s,t}(\BF_q).$

We now consider matrices of the form
 $$y_r=\begin{pmatrix} 1 & r \\0 & 1\end{pmatrix}.$$
A pair $(x,y)\in \pi^{-1}(s',u,2)$ if
 $$x=\begin{pmatrix} a & b \\ c & s'-a \end{pmatrix}$$
and $$a(s'-a)-bc=1, \ rc+s'=u .$$
  This implies that
 $$c=\fr{u-s'}{r}=\fr{\pm w}{r}, \ b=\fr{a(s'-a)-1}{c}.$$

Hence for a fixed $y_r$  there are at most $2q$ possible matrices
$x$ defined by the value of $a$ and by the sign of $c.$ Together
with \eqref{tau} we get
 $$|\pi^{-1}(s',u,2)| \leq 2q(q^2+q).$$  It follows from \eqref{i2}
that
 \beq\label{i3}
 |\tilde p^{-1}(s,2)|\leq 2q(q^2+q)2q\leq 5q^4.\ef

 {\bf(2.2)}. Assume that $t=-2.$  Similarly to  {\bf(2.1)}
 we get
 \beq\label{i4}
 |\tilde p^{-1}(s,2)|\leq 2q(q^2+q)2q\leq 5q^4.\ef

To complete the proof we may  take  $M_1=4$ and $M_2=10.$
\end{proof}

We proceed with the proof of the Proposition. By \thmref{qbig} and
\propref{genus}, the  fiber   $R_a=\rho_n^{-1}(a)$ of $\rho_n$ is
isomorphic to a general fiber of $X_{a-2}=r_n^{-1}(a-2)$ and is a
curve  of genus $g_n<G_n,$ where the bound $G_n$ depends only on
$n$. Moreover, it has at most  $2^{n}$ points at infinity and $2
\cdot 2^{n}$ points with $t^2=4.$  It does not have points of the
form $(t^2-2,t),$ since $\mu(t^2-2,t)=(2,t)$, which is is a fixed
point.

Let $A=R_a\cap\{t^2\ne 4\}$ and $B=R_a\cap\{t^2= 4\}.$

According to the Weil estimate, \beq\label{A} |A(\BF_q)| =q(1+\ep_3(n,q)),
\ef where
$$|\ep_3(n,q)|\leq \fr{1+2\sqrt{q}\cdot G_n+3 \cdot 2^{n}}{q}.$$

Hence, according to \lemref{schet}(1), \beq\label{A1} |\tilde
p^{-1}(A)(\BF_q)|=q(1+\ep_3(n,q))q^4(1+\ep_2):=q^5(1+\ep_4(n,q)),\ef
where
$$|\ep_4(n,q)|\leq |\ep_3(n,q)|+|\ep_2|+|\ep_3(n,q)|\cdot |\ep_2|=O\bigl(\fr{1}{\sqrt{q}}\bigr).$$

There are at most $2^{n+1}$ points in $B,$ thus by
\lemref{schet}(2), \beq\label{B1} |\tilde p^{-1}(B)(\BF_q)|\leq
2^{n+1}q^4M_2. \ef

Therefore, \beq\label{ro} |\a^{-1}(a)|=q^5(1+\ep_5(n,q)),\ef where
$$|\ep_5(n,q)|\leq |\ep_4(n,q)|+\fr{2^{n+1}M_2}{q} = O\bigl(\fr{1}{\sqrt{q}}\bigr).$$

Finally, from \eqref{gamma} and \eqref{ro} we obtain
$$|\g^{-1}(z)|=\fr{|\a^{-1}(a)|}{|\tau^{-1}(a)|}=\fr{q^5(1+\ep_5(n,q))}{q^2(1+\ep_1(q))}\\
=q^3\left(1+O\bigl(\fr{1}{\sqrt{q}}\bigr)\right),$$
as needed.
\end{proof}

We shall now show that \propref{equi.q} implies that the $n$-th
Engel word map is also almost equidistributed for the family of
groups $\PSL(2,q)$, where $q$ is odd.

Denote by $\bar{g}$ the image of $g \in \tG=\SL(2,q)$ in
$G=\PSL(2,q)$. Since $q$ is odd, one may identify $\bar{g}$ with the
pair $\{\pm g\}$.

Let $S' = \{g \in \tG: g \in S \text{ and } -g \in S \} \subseteq
S$. Then, by \propref{equi.q}{\it(i)}, $|S'| \leq |\tG|(1-2\ep)$.
Hence, if $\bar{S}'$ is the image of the set $S'$ in $G=\PSL(2,q)$,
then
\[
    |\bar{S}'| \leq |G|(1-2\ep).
\]

For $\bar{g}\in G=\PSL(2,q)$, denote
\[
   \bar{E}_n(\bar{g}) = \{(\bar{x},\bar{y}) \in G \times G: e_n(\bar{x},\bar{y})=\bar{g}
   \}.
\]
Observe that for any $x,y \in \tG$,
\[
    e_n(x,y)=e_n(-x,y)=e_n(x,-y)=e_n(-x,-y),
\]
thus
\[
    4\cdot \bar{E}_n(\bar{g}) = E_n(g) \cup E_n(-g),
\]
(this is a disjoint union) and so
\[
    \frac{|\bar{E}_n(\bar{g})|}{|G|} =
    \frac{|E_n(g)|+|E_n(-g)|}{2\cdot|\tG|}.
\]
Therefore, by \propref{equi.q}{\it(ii)}, for any $\bar{g} \in
\bar{S}'$ one has that
\[
    (1-\ep)|G| \leq \bar{E}_n(\bar{g}) \leq (1+\ep)|G|,
\]
completing the proof of \thmref{thm.equi}.


\section{Concluding remarks}\label{last}

The \emph{trace map} is an efficient way to translate an Algebraic
word problem on $\PSL(2,q)$ to the language of Geometry and
Dynamics, which has already been used fruitfully in \cite {BGK}. We
use it in this paper for studying the Engel words, but, actually,
the same could be done for any other word with the same dynamical
properties. Thus, one may ask the following questions:

\begin{quest}
What are the words for which the corresponding trace map $\psi(s,u,t)=(f_1(s,u,t),f_2(s,u,t),t)$
has the following property $(\star)$ for almost all $q$:

$(\star)$ For every $a\in \BF_q$ the set $\{f_1(s,u,t)=a\}$ is an
absolutely irreducible affine set.
\end{quest}

\begin{quest}
What are the words for which the trace map $\psi(s,u,t)=
(f_1(s,u,t),f_2(s,u,t),t)$ has an invariant plane $A$ and the curves
$\{\psi \bigm|_A=a\}$ are absolutely irreducible for a general $a\in
\BF_q$ and for almost all $q$?
\end{quest}

We believe that these two Questions are closely related to the
following variant of Shalev's Conjecture \ref{conj.surjall}:

\begin{conj}[Shalev]
Assume that $w=w(x,y)$ is not a power word, namely, it is not of the form $v(x,y)^m$ for some
$v \in F_2$ and $m \in \mathbb{N}$.
Then $w(G)=G$ for $G=\PSL(2,q)$.
\end{conj}

One can moreover ask these questions for finite simple non-abelian groups in general.

\begin{quest}
What is an analogue of the trace map for other finite simple
non-abelian groups? In particular, for the Suzuki groups $\Sz(q)$?
(see \cite[\S 4]{BGK}).
\end{quest}

Another interesting question is related to the explicit estimates
for $q$ in \propref{genus}. The genus of the curve $X$ given there
is very large, and this leads to an exponential bound for $q$, as a
function of $n$, for which $X(\BF_q)\ne \emptyset.$ On the other
hand, computer experiments using \textsc{Magma} demonstrate that
this estimate should be at most polynomial. It would be very
interesting to investigate $X$ and to understand this phenomenon.


\section*{Acknowledgments}
Bandman is supported in part by Ministry of Absorption (Israel), Israeli Academy of Sciences
and Minerva Foundation (through the Emmy Noether Research Institute of Mathematics).

Garion is supported by a European Post-doctoral Fellowship (EPDI),
during her stay at the Max-Planck-Institute for Mathematics (Bonn)
and the Institut des Hautes \'{E}tudes Scientifiques
(Bures-sur-Yvette).

This project started during the visit of Bandman and Garion to the
Max-Planck-Institute for Mathematics (Bonn) in 2009 and continued
during the visit of Grunewald to the Hebrew University of Jerusalem
and Bar-Ilan University (2010).

Bandman and Garion are most grateful to B. Kunyavskii for his
constant and very valuable help, to S. Vishkautsan and Eu. Plotkin
for numerous and useful comments. The authors are thankful to A.
Shalev for discussing his questions and conjectures with them. They
are also grateful to M. Larsen, A. Reznikov and V. Berkovich.

The authors are grateful to the referee for his valuable comments.

\vskip 0.2 cm
\emph{Fritz Grunewald has unexpectedly passed away in March 2010.
This project started as a joint project with him, and unfortunately it is published only after his death.
Fritz Grunewald has greatly inspired us and substantially influenced our work. He is deeply missed.}



\begin{thebibliography}{GKNP00}

\bibitem[AP]{AP}
Y.~Aubry and M.~Perret, {\em A Weil theorem for singular curves},
In: ``Arithmetic, Geometry and Coding Theory'', R.~Pellikaan,
M.~Perret, and S.~G.~Vl\u{a}du\c{t} (eds.), Walter de Gruyter,
Berlin--New York, 1996, pp.~1--7.


\bibitem[BGK]{BGK}
T. Bandman, F. Grunewald, B. Kunyavskii,
{\em Geometry and arithmetic of verbal dynamical systems on simple groups},
preprint available at arXiv:0809.0369.


\bibitem[Bl]{Bl}
H. Blau,
{\em A fixed-point theorem for central elements in quasisimple groups},
Proc. Amer. Math. Soc. {\bf 122} (1994), no. 1, 79--84.


\bibitem[Bo]{Bo}
A. Borel, {\em On free subgroups of semisimple groups},
Enseign. Math. (2) {\bf 29} (1983), no. 1-2, 151--164.


\bibitem[Mag]{Magma}
W. Bosma, J. Cannon, C. Playoust,
{\em The Magma algebra system. I. The user language}, J. Symbolic Comput.
\textbf{24} (1997), no.\ 3--4, 235--265


\bibitem[CMS] {CMS}
J. Cossey, Sh. O. Macdonald, A. P. Street,
{\em On the laws of certain finite groups},
J. Australian  Math. Soc. {\bf 11} (1970), 441--489.


\bibitem[DS]{DS}
V.I. Danilov, V.V Shokurov
{\em Algebraic curves, algebraic manifolds and schemes},
Encyclopedia of Math. Sciences, vol. 23, Springer, 1998.


\bibitem[Do]{Do}
L. Dornhoff, {\em Group Representation Theory}, Part A, Marcel Dekker, 1971.


\bibitem[EG]{EG}
E.W. Ellers, N. Gordeev, {\em On the conjectures of J. Thompson and O. Ore},
Trans. Amer. Math. Soc. {\bf 350} (1998), 3657--3671.


\bibitem [Fr] {Fr}
R. Fricke, {\em \"Uber die Theorie der automorphen Modulgruppen},
Nachr. Akad. Wiss. G\"ottingen (1896), 91--101.


\bibitem [FK] {FK}
R. Fricke, F. Klein, {\em Vorlesungen der automorphen Funktionen},
vol.~1--2, Teubner, Leipzig, 1897, 1912.


\bibitem[GS]{GS}
S. Garion, A. Shalev,
{\em Commutator maps, measure preservation, and $T$-systems},
Trans. Amer. Math. Soc. \textbf{361} (2009), no. 9, 4631--4651.


\bibitem[GL]{GL}
S.~R.~Ghorpade and G.~Lachaud, {\em Etale cohomology,
Lefschetz theorems and number of points of singular varieties over finite fields},
Moscow Math.\ J.\ {\bf 2} (2002), 589--631, and CORRIGENDA AND ADDENDA.

\bibitem[Go]{Go}
W. Goldman, {\em An exposition of results of Fricke and Vogt},
preprint available at
http://www.math.umd.edu/\~{}wmg/publications.html .


\bibitem[Ha]{Ha}
R. Hartshorne, {\em Algebraic geometry}, Springer-Verlag, NY, 1977.


\bibitem[KL]{KL}
W.M. Kantor, A. Lubotzky, {\em The probability of generating a finite classical group},
Geom. Ded. {\bf 36} (1990), 67--87.



\bibitem[La]{La}
M. Larsen, {\em Word maps have large image}, Israel J. Math. {\bf 139} (2004), 149--156.


\bibitem[LS]{LS}
M. Larsen, A. Shalev, {\em Word maps and Waring type problems},
J. Amer. Math. Soc.  \textbf{22} (2009),  no. 2, 437--466.


\bibitem[LST]{LST}
M. Larsen, A. Shalev, P.H. Tiep, {\em Waring problem for finite
simple groups}, preprint.


\bibitem[LY]{LY}
D. Leep and C. Yeomans, {\em The number of points on a singular
curve over a finite field}, Arch.\ Math.\ (Basel) {\bf 63} (1994),
420--426.


\bibitem[LOST]{LOST}
M.W. Liebeck, E.A. O'Brien, A. Shalev, P.H. Tiep, {\em The Ore conjecture},
J. European Math. Soc. {\bf 12} (2010), 939--1008.


\bibitem[Ma]{Ma}
W. Magnus, {\it Rings of Fricke characters and automorphisms groups
of free groups}, Math. Z. {\bf 170} (1960), 91--102.


\bibitem[MW]{MW}
D. McCullough, M. Wanderley, {\em Writing Elements of $\PSL(2,q)$ as
commutators}, preprint.


\bibitem[Or]{Or}
O. Ore, {\em Some remarks on commutators}, Proc. Amer. Math. Soc. {\bf 2}
(1951), 307--314.


\bibitem[Sa]{Sa}
P. Samuel, {\em M\'ethodes  d'alg\`ebre abstraite en g\'eom\'etrie alg\'ebrique},
Springer-Verlag, NY, 1967.


\bibitem[Se]{Se}
D. Segal, {\em Words: notes on verbal width in groups},
London Mathematical Society Lecture Note Series {\bf 361}, Cambridge University Press, Cambridge, 2009.


\bibitem[Sh07]{Sh07}
A. Shalev, {\em Commutators, words, conjugacy classes and character methods},
Turkish J. Math. {\bf 31} (2007), suppl., 131--148.


\bibitem[Sh09]{Sh09}
A. Shalev, {\em Word maps, conjugacy classes, and a non-commutative Waring-type theorem},
Annals of Math. {\bf 170} (2009), 1383--1416.


\bibitem[Th]{Th}
R.C. Thompson, {\em Commutators in the special and general linear
groups}, Trans. Amer. Math. Soc. {\bf 101} (1961), 16--33.


\bibitem [Vo] {Vo}
H. Vogt, {\em Sur les invariants fundamentaux des equations
diff\'erentielles lin\'eaires du second ordre}, Ann. Sci. E.N.S,
3-i\`eme S\'er. {\bf 4} (1889), Suppl. S.3--S.70.

\end{thebibliography}
\end{document}